\documentclass[12pt,letterpaper]{amsart}
\usepackage[utf8x]{inputenc}

\usepackage{amsfonts,amssymb,amstext,amsmath,amscd,latexsym,amsthm}

\usepackage{answers}
\Newassociation{sol}{Solution}{ans}

\usepackage{pstricks}
\usepackage{pst-plot}

\theoremstyle{plain}
\newtheorem{theorem}{Theorem}
\theoremstyle{definition}
\newtheorem{definition}{Definition}
\newtheorem{problem}{Problem}
\theoremstyle{remark}
\newtheorem{example}{Example}
\newtheorem{exercise}{Exercise}
\newtheorem{remark}{Remark}

\author{J. R. Arteaga, M. Malakhaltsev}
\title[$G$-structures and differential equations]{Ideas of E.~Cartan and S.~Lie in modern geometry: $G$-structures and differential equations. 
Lecture 1}

\begin{document}
\Opensolutionfile{ans}[Lecture1_answers]
\maketitle

\fbox{\fbox{\parbox{5.5in}{
\textbf{Problem:}\\
How to reduce to a simplest equation one ODE.
}}}
\vspace{1cm}

Our purpose of this mini-curse is to explain some ideas of E. Cartan and S. Lie when we study differential geometry, particularly we will to explain the Cartan reduction method. The Cartan reduction method is a technique in Differential Geometry for determining whether two geometrical structure are the same up to a diffeomorphism. This method use new tools of differential geometry as principal bundles, $G$-structures and jets theory. We start with an example of a $G$-structure: the $3$-webs in $\mathbb{R}^{2}$. Here we use the Cartan method  to classify the differential equations but not to resolve.   
This is a classification can be a weak classification in the sense of not involving all the structural invariants.

\section*{$3$-Webs in $\mathbb{R}^{2}$}
\begin{definition}\label{Def:3-web}
A collection of three foliations of $\mathbb{R}^{2}$, $\mathcal{L} = \{L_{1}, L_{2}, L_{3}\}$ defined in an open set  $U$ in the plane, such that pairwise are transverse  is called a $3$-web.
\end{definition}

\begin{example}\label{Exa:example1_3-web}
Let $\mathcal{L} = \{L_{1}, L_{2}, L_{3}\}$ be the $3$-web  in $\mathbb{R}^{2}$ where
\begin{equation}
\begin{cases}
L_{1} = \{(x,y) \mid x= \text{const.}\}\\
L_{2} = \{(x,y) \mid y= \text{const.}\}\\
L_{3} = \{(x,y) \mid y-x= \text{const.}\}
\end{cases}
\end{equation}
\end{example}

\subsection*{Associated $3$-web to one ODE}

\begin{theorem}\label{Th:theorem_1}

For every ordinary differential equation ODE of first order,
\begin{equation}\label{Eq:EDO_1}
\dfrac{dy}{dx} = F(x,y) 
\end{equation}
where  $F(x,y)$ is a smooth function defined in an open set  $U \subseteq \mathbb{R}^{2}$, such that  $F(x,y) \neq 0$ for all  $(x,y)\in U$,   we can always associated  a  $3$-web defined in $W$ for some  $W \subseteq U$.
\end{theorem}

\begin{proof}
For any smooth function $F(x,y)$, s.t. $F(x,y)\neq 0$ in $U$ we can associate the $3$-web $\mathcal{L} = \{L_{1}, L_{2}, L_{3}\}$ where
\begin{equation}
\begin{cases}
L_{1} = \{(x,y) \mid x= \text{const.}\}\\
L_{2} = \{(x,y) \mid y= \text{const.}\}\\
L_{3} = \{(x,y) \mid y-f(x)= \text{const.}\}
\end{cases}
\end{equation}
where  $f(x)$ is the family of integral curves  of  \eqref{Eq:EDO_1} in $U$. The Picard-Lindelöff theorem guarantees us the existence of this family in any $W\subseteq  U$.
\end{proof}

\begin{remark}\label{Re:remark1}
The ODE  \eqref{Eq:EDO_1} where $F(x,y)$ satisfies the conditions of theorem \ref{Th:theorem_1} we called one ODE $3$-web type.
\end{remark}

\begin{example}\label{Exa:example2}
Let $U = \{(x,y) \mid x > 1\}$ be an open set of  $\mathbb{R}^{2}$. Consider the following ODE in $U$,
\begin{equation}
\dfrac{dy}{dx} = 1-x
\end{equation}
In $U$ we define the associated $3$-web $\mathcal{L} = \{L_{1}, L_{2}, L_{3}\}$ where
\begin{equation}
\begin{cases}
L_{1} = \{(x,y) \mid x= \text{const.}\}\\
L_{2} = \{(x,y) \mid y= \text{const.}\}\\
L_{3} = \{(x,y) \mid y-\left( x - \dfrac{x^{2}}{2} \right)= \text{const.}\}
\end{cases}
\end{equation}
In the Figure \ref{Fig:example_2} we can see that effectively $\mathcal{L}$ is a $3$-web in $U$.

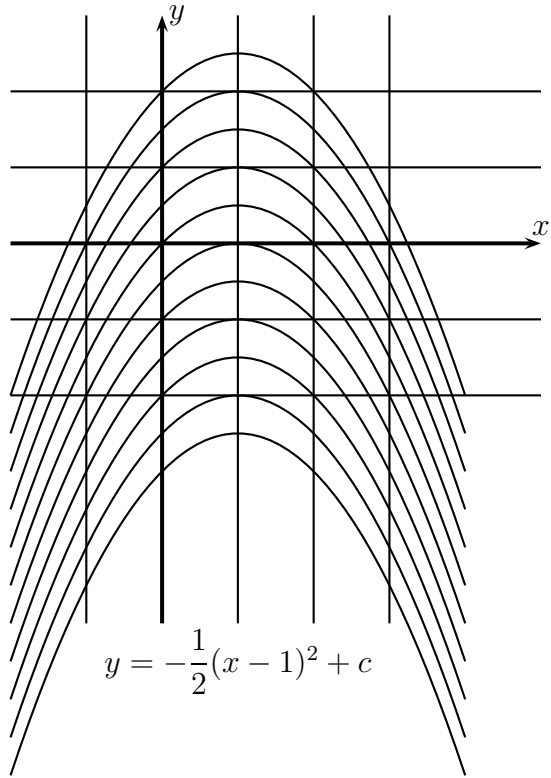
\begin{figure}[h]
\begin{center}
\psset{xunit=1cm,yunit=1cm}
\begin{pspicture}(-5,-7)(7,3)
\psline[linewidth=1.5pt]{->}(-2,0)(5,0)\rput(5,0.2){$x$}
\psline[linewidth=1.5pt]{->}(0,-5)(0,3)\rput(0.2,3){$y$}
\rput(1,-5.5){$y=-\dfrac{1}{2}(x-1)^{2}+c$}
\psline{-}(-2,1)(5,1)
\psline{-}(-2,2)(5,2)
\psline{-}(-2,-1)(5,-1)
\psline{-}(-2,-2)(5,-2)
\psline{-}(1,-5)(1,3)
\psline{-}(2,-5)(2,3)
\psline{-}(3,-5)(3,3)
\psline{-}(-1,-5)(-1,3)
\psplot{-2}{4}{-0.5 x 1 sub 2 exp mul 0.5 add} %
\psplot{-2}{4}{-0.5 x 1 sub 2 exp mul 1 add} %
\psplot{-2}{4}{-0.5 x 1 sub 2 exp mul 1.5 add} %
\psplot{-2}{4}{-0.5 x 1 sub 2 exp mul 2 add} %
\psplot{-2}{4}{-0.5 x 1 sub 2 exp mul 2.5 add} %
\psplot{-2}{4}{-0.5 x 1 sub 2 exp mul 0 add} %
\psplot{-2}{4}{-0.5 x 1 sub 2 exp mul -0.5 add} %
\psplot{-2}{4}{-0.5 x 1 sub 2 exp mul -1 add} %
\psplot{-2}{4}{-0.5 x 1 sub 2 exp mul -1.5 add} %
\psplot{-2}{4}{-0.5 x 1 sub 2 exp mul -2 add} %
\psplot{-2}{4}{-0.5 x 1 sub 2 exp mul -2.5 add} %
\end{pspicture}
\caption{Example 2, $3$-web associated to $y' = 1-x$}
\label{Fig:example_2}
\end{center}
\end{figure}
\end{example}

\begin{exercise}\label{Exe:exercise1}
Given the ODE $ y '= y $, determine an open set $ U $ of the plane where you can define a $ 3 $-web associated with it. Write the $ 3 $-web explicitly and draw a picture of it.
\begin{sol}
For example $ U = \{(x,y)\mid y > 0\} $, 
$ L_{1} = \{(x,y) \mid x= \text{const.}\} $, 
$ L_{2} = \{(x,y) \mid y= \text{const.}\} $, 
$ L_{3} = \{(x,y) \mid y-c e^{x}=0, \, \text{where}\,  c=\text{const.}\} $
\end{sol}
\end{exercise}

\subsection*{Equivalence of $3$-webs}

\begin{definition}\label{Def:web_equivalence}
We say that two $3$-web $\mathcal{L}_ {1}$ and $\mathcal {L}_ {2}$ defined in $U_ {1}$ and $ U_{2}$ respectively, are equivalent if exists a diffeomorphism $\varphi$,
\begin{equation}
\varphi : U_{1} \longrightarrow U_{2}
\end{equation}
such that $\varphi (\mathcal{L}_{1}) = \mathcal{L}_{2}$.
In this case we will write $\mathcal{L}_{1} \sim \mathcal{L}_{2}$.
\end{definition}

The diffeomorphism $\varphi$ we can write as, 
\begin{equation}
\begin{split}
\varphi :& \, U_{1} \longrightarrow U_{2}\\
& (x,y) \mapsto \varphi(x,y)= (\tilde{x}, \tilde{y})=(\alpha(x,y), \beta(x,y))
\end{split}
\end{equation}

\subsection*{Adapted coframe to a $3$-web}

Let  $\mathbb{R}^{2}$, $\mathcal{L} = \{L_{1}, L_{2}, L_{3}\}$ be a $3$-web defined in $U$. Let 
$\vec{u}_{i}(x,y)$ be a vector field that has $L_{i}$ as its integral curves, i.e. $L_{i}$ is the flow of $\vec{u}_{i}$ for $i=1,2,3$. 
\begin{equation}
\begin{cases}
L_{1} = \text{flow of\quad } \vec{u}_{1}\\
L_{2} = \text{flow of\quad } \vec{u}_{2}\\
L_{3} = \text{flow of\quad } \vec{u}_{3}
\end{cases}
\end{equation}
This means that if $\vec{r}_{i}(t)$ is a parametrization of $L_{i}$, then 
\begin{equation}
\dfrac{d}{dt}\vec{r}_{i}(t) = \lambda_{i}\vec{u}_{i}(x(t),y(t)), \quad \text{for any \quad} \lambda_{i}\in \mathbb{R} \backslash 0
\end{equation}
By definition \ref{Def:3-web} for any pair $\{\vec{u}_{i}, \vec{u}_{j}\}$, ($i,j \in \{1,2,3\}$, $i\neq j$), $\vec{u}_{i}$ and $\vec{u}_{j}$ are linear independent in each point $(x,y) \in U$, then we can give the following definition.

\begin{definition}[Adapted coframe to $\mathcal{L}$]
The coframe $ \{\eta^{1}, \eta^{2}\} $ where,
\begin{equation}\label{Eq:associate_frame}
\begin{cases}
\eta^{1} \text{ \,  annihilates \,} \vec{u}_{2}, \quad \text{i.e.} \quad \eta^{1}(\vec{u}_{2})=0\\
\eta^{2}\text{ \,  annihilates \,} \vec{u}_{1}, \quad \text{i.e.} \quad \eta^{2}(\vec{u}_{1})=0\\
\eta^{3}= \eta^{1} - \eta^{2} , \text{ \,  annihilates \,} \vec{u}_{3},  \quad \text{i.e.} \quad \eta^{3}(\vec{u}_{3})=0\\
\end{cases}
\end{equation}
is called an \emph{adapted coframe} to a $3$-web $\mathcal{L} = \{L_{1}, L_{2}, L_{3}\}$.
\end{definition}

\begin{example}\label{Exa_example3} 
Find an adapted coframe to the ODE,
\begin{equation}
y' = 1 -x
\end{equation}
Let $U = \{(x,y) \in \mathbb{R}^{2} \mid x > 1\}$. The associated $ 3 $-web $\mathcal{L} = \{L_{1}, L_{2}, L_{3}\}$ described in the Example \ref{Exa:example2} we can re-write in terms of vector fields as  follows:
\begin{equation}
\begin{cases}
L_{1} = span \{\vec{u}_{1}\} \quad \text{where}\quad
\vec{u}_{1} = \dfrac{\partial}{\partial x}  
\\
L_{2} = span \{\vec{u}_{2}\} \quad \text{where}\quad
\vec{u}_{2} = (1-x)\dfrac{\partial}{\partial y}
\\
L_{3} = span \{\vec{u}_{1} + \vec{u}_{2}\} \quad \text{where}\quad
\vec{u}_{3} = \vec{u_{1}}+\vec{u}_{2} = 
\dfrac{\partial}{\partial x}  + (1-x)\dfrac{\partial}{\partial y}
\end{cases}
\end{equation}
An adapted coframe $ \eta = \{\eta^{1}, \eta^{2} \} $ to this ODE is
\begin{equation}
\begin{cases}
\eta^{1}= (1-x) dx\\
\eta^{2}= dy\\
\eta^{3}= dx - dy
\end{cases}
\end{equation}
\end{example}

\begin{exercise}\label{Exe_exercise2} 
Find one associate frame $ \eta = \{\eta^{1}, \eta^{2} \} $ to the ODE,
\begin{equation}
y' = y
\end{equation}
\begin{sol}
We can take as associated coframe $ \eta = \{\eta^{1}, \eta^{2} \} $  the following,
\begin{equation}
\begin{cases}
\eta^{1}= y dx\\
\eta^{2}= dy\\
\eta^{3}= \eta^{1} - \eta^{2} = ydx - dy
\end{cases}
\end{equation}
\end{sol}
\end{exercise}

\section*{Equivalence of ODE's web-type}
\begin{problem}
Given two ODE,
\begin{equation}
\dfrac{dy}{dx} = F_{1}(x,y), \qquad 
\dfrac{dy}{dx} = F_{2}(x,y),
\end{equation}
under what conditions we can say that they are equivalent?
\end{problem}

Suppose that we get the following definition:

\begin{definition}[Bad definition]
Two ODE
\begin{equation}
\dfrac{dy}{dx} = F_{1}(x,y), \qquad 
\dfrac{dy}{dx} = F_{2}(x,y),
\end{equation}
are equivalents if there exists a change of coordinates (diffeomorphism $\phi$),
  \begin{equation}
  \begin{split}
 & \phi : (x,y) \mapsto (\tilde{x}, \tilde{y}),\\
 & \tilde{x} = \alpha (x,y)\\
 & \tilde{y} = \beta (x,y)
  \end{split}
\end{equation}
such that $ \phi $ sends the associate $ 3 $-web of one to the associate $ 3 $-web of the other one.
\end{definition}

This definition is a bad definition because we do not know any invariant of one ODE under a coordinates change. For this reason we must to know some invariant of a $ 3 $-web associated to one ODE.
We will approach to solve this problem  from the viewpoint of Cartan. 

\subsection*{Adapted coframe to one ODE}

\begin{definition}[Adapted coframe to one ODE.]
Let
\begin{equation} \label{Eq:ODE}
\dfrac{dy}{dx} = F(x,y), 
\end{equation}
be one ODE $3$-web type.
A coframe $ \{\eta^{1}, \eta^{2}\} $ of the plane $\mathbb{R}^{2}$ such that,
\begin{equation}\label{Eq:adapted_coframe_ODE}
\boxed{
\begin{cases}
\eta^{1}= F(x,y) dx\\
\eta^{2}= dy\\
\eta^{3}= \eta^{1} - \eta^{2} = F(x,y)dx - dy
\end{cases}
}
\end{equation}
is called an \emph{adapted coframe} to one ODE  \eqref{Eq:ODE}
\end{definition}

\subsection*{The Blaschke-Chern curvature form}

Suppose that we have one EDO $ 3 $-web types  with two different adapted coframes \eqref{Eq:adapted_coframe_ODE},
\begin{equation}
\begin{split}
\dfrac{dy}{dx} = F(x,y) \rightsquigarrow
\begin{cases}
\eta^{1} = F(x,y)dx \\
\eta^{2} = dy\\
\eta^{3} = \eta^{1} - \eta^{2} = F(x,y)dx - dy
\end{cases}
\\
\dfrac{dy}{dx} = F(x,y) \rightsquigarrow
\begin{cases}
\underline{\eta}^{1} = F(x,y)dx \\
\underline{\eta}^{2} = dy\\
\underline{\eta}^{3} = \underline{\eta}^{1} - \underline{\eta}^{2} = F(x,y)dx - dy
\end{cases}
\end{split}
\end{equation}

We need to determine how unique our choice of adapted coframe is?. 
Any other adapted coframe  change by the rule,
\begin{equation}\label{Eq:equation_19}
\left(
\begin{matrix}
\tilde{\eta}^{1}\\
\tilde{\eta}^{2}
\end{matrix}
\right)
=
\left(
\begin{matrix}
\alpha & 0 \\
0 & \beta
\end{matrix}
\right)
\left(
\begin{matrix}
\eta^{1}\\
\eta^{2}
\end{matrix}
\right)
\end{equation}
for any non vanishing functions $ \alpha = \alpha (x,y) $ and $ \beta = \beta (x,y) $. 
But the third condition of an adapted coframe must be satisfy,
\begin{equation}\label{Eq:equation_20}
\tilde{\eta^{3}} = \gamma \eta \Rightarrow
\tilde{\eta^{1}} - \tilde{\eta^{2}}  = \gamma 
(\eta^{1}  - \eta^{2} )
\end{equation}
for any non vanishing function $ \gamma = \gamma (x,y) $.
Solving equations \eqref{Eq:equation_19} and \eqref{Eq:equation_20} we have that 
\begin{equation}
\alpha = \beta = \gamma
\end{equation}
Therefore two adapted coframes to the same ODE satisfy,
\begin{equation}\label{Eq:scalar_change}
\boxed{
\left(
\begin{matrix}
\tilde{\eta}^{1}\\
\tilde{\eta}^{2}
\end{matrix}
\right)
=
\left(
\begin{matrix}
\alpha & 0 \\
0 & \alpha
\end{matrix}
\right)
\left(
\begin{matrix}
\eta^{1}\\
\eta^{2}
\end{matrix}
\right)
}
\end{equation}
for non-vanishing function $ \alpha = \alpha (x,y)$. 

\subsection*{The space of  adapted coframes to one ODE}

Let $B$ be the space of all adapted coframes to one ODE given by the $1$-forms  
\begin{equation}
\eta = (\eta^{1}, \eta^{2}) =
\begin{cases}
\eta^{1} = F(x,y)dx \\
\eta^{2} = dy\\
\eta^{3} = \eta^{1} - \eta^{2} = F(x,y)dx - dy
\end{cases}
\end{equation}

Let us denote by $q=(\vec{x},\eta)$ a point of $B$, where $\vec{x}=(x,y)$ is a point of the plane $\mathbb{R}^{2}$ and $\eta = \{\eta^{1}, \eta^{2}\}$ is an adapted coframe \eqref{Eq:adapted_coframe_ODE} of a fixed ODE. Two points $q_{1}, q_{2} \in B$ with $p$ fixed satisfy $q_{2} = g\cdot q_{1}$, where $g$ is a scalar matrix, $g\in E(2)$. The set of the scalar matrices $E(2)$ is a group  under standard multiplication of matrices and it is a sub-group of the general linear group $GL(2)$.  

$B$ is a manifold and precisely meets the definition of a principal $G$-bundle where the group actions is $E(2)$ and projection $\pi : B \longrightarrow \mathbb{R}^{2}$.

On the manifold $B$ we can take coordinates $(x,y,\alpha)$ and we will consider the coframe field: 
\begin{equation}
\boxed{
\theta^{i} = \left(g^{-1}\right)^{i}_{s}\eta^{s} 
}
\qquad
\boxed{
\begin{split}
&\theta^{1} = \alpha^{-1}\eta^1 = \alpha^{-1} F(x,y) dx
\\
&\theta^{2} = \alpha^{-1}\eta^2 = \alpha^{-1} dy 
\\
&\theta^{3} = d\alpha
\end{split}
}
\label{eq:coframe_field_on_B}
\end{equation}

\begin{theorem}
There exists a unique $1$-form $\varphi$ such that
\begin{equation}\label{Eq:equation_19_1}
\left(
\begin{matrix}
d \theta^{1}\\
d \theta^{2}
\end{matrix}
\right)
=
\left(
\begin{matrix}
\varphi & 0 \\
0 & \varphi
\end{matrix}
\right)
\wedge
\left(
\begin{matrix}
\theta^{1}\\
\theta^{2}
\end{matrix}
\right)
\end{equation}
are satisfied.
\end{theorem}

\begin{proof}
Calculate
\begin{equation}
\begin{split}
& d\theta^{1} = -\dfrac{1}{\alpha^{2}} F d\alpha \wedge dx + \dfrac{1}{\alpha} \partial_y F dy \wedge dx = 
\left( -\dfrac{1}{\alpha}d\alpha +\dfrac{F_{y}}{F}dy\right)\wedge \dfrac{1}{\alpha}F dx
\\
& d\theta^{2} =  -\dfrac{1}{\alpha^{2}} d\alpha \wedge dy =
 \left( -\dfrac{1}{\alpha}d\alpha +\dfrac{F_{y}}{F}dy\right)\wedge \dfrac{1}{\alpha} dy
\end{split}
\end{equation}
Now we take
\begin{equation}
\boxed{
\varphi = -\frac{1}{\alpha} d\alpha + \frac{\partial_y F}{F} dy  
}
\label{eq:form_theta}
\end{equation} 
One can easily check that 
\begin{equation*}
d\theta^{1} = \varphi \wedge \theta^{1} \text{ and }
d\theta^{2}= \varphi \wedge \theta^{2}. 
\end{equation*}
This proves the theorem.
\end{proof}

\begin{definition}
Given an adapted coframe $\eta$ in the manifold of adapted coframes of one ODE $B$ the $1$-forms $\theta^{1}$ and $\theta^{2}$ that satisfy \eqref{eq:coframe_field_on_B} are called the \emph{tautological forms}. 
Any choice of a tautological forms (coframe field $\theta = (\theta^{1}, \theta^{2})$) give as an unique $1$-form 
\begin{equation}
\omega =
\left(
\begin{matrix}
\varphi & 0 \\
0 & \varphi
\end{matrix}
\right)
\end{equation}
called \emph{associated \emph{connection form}}. The equations \eqref{Eq:equation_19_1} are called the \emph{structural equations}.
\end{definition}

Now we have 
\begin{equation*}
d\varphi = d\left( \frac{\partial_y F}{F} dy\right) = 
\frac{F\partial_{xy}F - \partial_x F \partial_y F}{F^2} dx \wedge dy.  
\end{equation*}
From this follows that
\begin{equation}
d\eta = K \theta^{1}\wedge \theta^{2},
\end{equation}
where
\begin{equation}
K = \alpha^{2}\left(\frac{F\partial_{xy}F - \partial_x F \partial_y F}{F^3} \right)
\label{eq:Blaschke_Chern_curvature}
\end{equation}

\begin{definition}
Given a $3$-web $\mathbb{R}^{2}$, $\mathcal{L} = \{L_{1},L_{2},L_{3}\}$ with associated tautological forms 
$\{ \theta^{1}, \theta^{2} \}$ such that the function $K$ which satisfy
\begin{equation}
d\eta = K \theta^{1} \wedge \theta^{2},
\end{equation}
is called the \emph{Blaschke-Chern curvature} for the $3$-web.
\end{definition}

\begin{example}\label{Exa:example4}
Calculate the Blaschke-Chern curvature for 
\begin{equation}
\dfrac{dy}{dx}=1
\end{equation}
Using  \eqref{eq:Blaschke_Chern_curvature} we obtain that $K=0$.
\end{example}

\begin{example}\label{Exa:example5}
Calculate the Blaschke-Chern curvature for 
\begin{equation}
\dfrac{dy}{dx}=x+y
\end{equation}
By \eqref{eq:Blaschke_Chern_curvature}, we have that $K = -\frac{1}{(x+y)^3}$, 
therefore the ODEs $\dfrac{dy}{dx}=1$ and $\dfrac{dy}{dx}=x+y$ are not equivalent.
\end{example}

\begin{definition}[Good definition]\label{Def:good_definition}
Two ODE
\begin{equation}
\dfrac{dy}{dx} = F_{1}(x,y), \qquad 
\dfrac{dy}{dx} = F_{2}(x,y),
\end{equation}
are equivalents if they have the same Chern-Blaschke curvature. Always there exists a change of coordinates  $ \phi $ such that $\phi$ sends the associate $ 3 $-web of one to the associate $ 3 $-web of the other one.
\end{definition}

\begin{example}\label{Exa:example_5}
Decide if the ODEs:
\begin{equation}
\dfrac{dy}{dx} = 1-x, \qquad
\dfrac{dy}{dx} = xe^{-y}
\end{equation}
are equivalents in certain domains. Find a diffeomorphism $\phi$  such that $\phi$ sends the associate $ 3 $-web of one to the associate $ 3 $-web of the other one. 

Effectively the ODEs are equivalents up to the Chern-Blaschke curvature that is $K=0$. Then we can take
\begin{equation}
\begin{split}
U_{1} &= \{(x,y)\in \mathbb{R}^{2} \mid x>1 , y>0\}\\
U_{2} &= \{(x,y)\in \mathbb{R}^{2} \mid x>0 , y>0\}\\
(\tilde{x}, \tilde{y} ) &= \phi (x,y) = (x-1, \ln y)
\end{split}
\end{equation}

\begin{equation}
\begin{cases}
L_{1} = \{(x,y) \mid x= \text{const.}\}\\
L_{2} = \{(x,y) \mid y= \text{const.}\}\\
L_{3} = \{(x,y) \mid y=\left( x - \dfrac{x^{2}}{2} \right) +c_{1}\}
\end{cases}
\quad
\begin{cases}
L_{1} = \{(x,y) \mid x= \text{const.}\}\\
L_{2} = \{(x,y) \mid y= \text{const.}\}\\
L_{3} = \{(x,y) \mid y = \ln \left(c_{2}-\dfrac{x^{2}}{2}\right)\}
\end{cases}
\end{equation}
\end{example}

\begin{remark}
\begin{enumerate}
\item
The definition \ref{Def:good_definition} determines whether two EDO are the same up to a diffeomorphism. This means that we only we require one differential invariant, the Blaschke-Chern curvature but no all possible invariants.
\item
There is not a general method for to find the existing isomorphism. In the most cases we can find it by simple inspection.
\end{enumerate}
\end{remark}

\begin{exercise}\label{Exe:exercise3}
Decide if the ODEs:
\begin{equation}
\dfrac{dy}{dx} = xy \qquad
\dfrac{dy}{dx} = y
\end{equation}
are equivalents in certain domains. Find a diffeomorphism $\phi$  such that $\phi$ sends the associate $ 3 $-web of one to the associate $ 3 $-web of the other one. 
\begin{sol}
$U_{1}=U_{2}=\{(x,y)\in \mathbb{R}^{2} \mid x>0\}$, $K=0$ and 
$(\tilde{x}, \tilde{y} ) = \phi (x,y) = (x^{2}, y^{2})$
\end{sol}
\end{exercise}

\section*{Summary of Lecture 1}

The main problem is ``How to reduce to a simplest one ODE''.
\begin{itemize}
\item
Some ODE of first order we can reduce to a simplest are the ODE web-type. These equations are which we can associate a $3$-web.
\item
A $3$-web in $\mathbb{R}^{2}$ is a set of three foliation that are transverse, and is an example of a $G$-structure. 
\item
For any $3$-web we can calculated its Blaschke-Chern curvature form.
\item
Finally, two ODEs are equivalents if have the Blaschke-Chern curvature and this equivalence is up to a diffeomorphism. 
\end{itemize}

\textbf{What will we do in the next lecture?}. In the next lecture we will define what is $G$-structure in general a how we can adapted a coframe filed for a given geometrical structure.

\Closesolutionfile{ans}
\section*{Answers to exercises}
\input{Lecture1_answers}


\end{document}